\documentclass[10pt, a4paper]{amsart} 
\usepackage[left=2.50cm, right=2.50cm, top=2.50cm, bottom=2.50cm]{geometry} 
\usepackage{amscd, amsfonts, amsmath, amssymb, amsthm, arydshln, fixmath, graphicx, mathrsfs, overpic, tikz, tikz-cd}
\usepackage{hyperref}
\usepackage[all]{xy} 
\usepackage{tikz}
\usepackage{tikz}
\usetikzlibrary{arrows}
\usetikzlibrary{decorations.pathmorphing}
\usepackage{dsfont}
\makeindex

\theoremstyle{definition} 
\newtheorem{Unity}{Unity}[section] 
\newtheorem*{Definition*}{Definition} 
\newtheorem{Definition}[Unity]{Definition} 

\theoremstyle{plain} 
\newtheorem*{Theorem*}{Theorem}
\newtheorem{Theorem}[Unity]{Theorem}
\newtheorem{Proposition}[Unity]{Proposition}
\newtheorem{Corollary}[Unity]{Corollary}
\newtheorem{Lemma}[Unity]{Lemma}

\theoremstyle{remark} 
\newtheorem*{Remark*}{Remark}
\newtheorem{Remark}[Unity]{Remark}

\numberwithin{Unity}{section}

\newcommand{\Hom}{\mathrm{Hom}}

\newcommand{\Rep}{\mathrm{Rep}}

\newcommand{\Spec}{\mathrm{Spec\,}}

\newcommand{\Coh}{\mathfrak{Coh}}
\newcommand{\Qcoh}{\mathfrak{Qcoh}}
\newcommand{\Vect}{\mathfrak{Vect}}

\newcommand{\Ext}{\mathrm{Ext}}
\newcommand{\depth}{\mathrm{depth}}
\newcommand{\GLF}{\mathrm{GLF}}
\newcommand{\codim}{\mathrm{codim}}
\newcommand{\chara}{\mathrm{char\,}}

\begin{document}

\title{The Lefschetz Type Theorem For Fundamental Group Schemes}
\author{Lingguang Li}
\address{School of Mathematical Sciences,
Key Laboratory of Intelligent Computing and Applications (Tongji University), Ministry of Education, Shanghai 200092, CHINA}
\email{LiLg@tongji.edu.cn}
\author{Niantao Tian}
\address{School of Mathematical Sciences,
Key Laboratory of Intelligent Computing and Applications (Tongji University), Ministry of Education, Shanghai 200092, CHINA}
\email{tianniantao@tongji.edu.cn}
\begin{abstract} Let $k$ be a field, $X$ a connected scheme proper over $k$, $D\subsetneq X$ an ample effective connected divisor, $x\in D(k)$. For Tannakian categories $\mathcal{C}_X$ and $\mathcal{C}_D$ whose objects consist of vector bundles on $X$ and $D$ respectively, we establish general Tannakian criteria for the natural homomorphism \(\pi(\mathcal{C}_D,x)\to \pi(\mathcal{C}_X,x)\) to be faithfully flat, a closed immersion, or an isomorphism. As applications, under Langer type positivity assumptions, we prove that \(\pi^{\ast}(D,x)\longrightarrow \pi^{\ast}(X,x)\) is an isomorphism for $\ast\in\{S,N,EN,F, EF,Loc,ELoc,\acute{e}t,E\acute{e}t,uni\}$ over perfect fields. 
\end{abstract}
\maketitle
\tableofcontents

\section{Introduction}

In algebraic geometry and arithmetic geometry, understanding the behavior of fundamental groups under geometric operations is a central theme. Among the foundational results in this direction is the Lefschetz hyperplane theorem. In its classical form for the \'etale fundamental group, Grothendieck \cite{Gro60} proved that if $X$ is a smooth projective variety over an algebraically closed field $k$ and $D \subsetneq X$ is a smooth ample divisor, then the natural homomorphism \(\pi_1^{\acute{e}t}(D, \bar{x}) \to \pi_1^{\acute{e}t}(X, \bar{x})\) is an isomorphism for $\dim X \ge 3$. This theorem has had far-reaching consequences in both topology and arithmetic geometry.

With the development of the theory of fundamental group schemes through Tannakian categories, several analogues of the classical fundamental group have been introduced by considering different categories of vector bundles. Nori \cite{Nor76, Nor82} introduced the Nori fundamental group scheme $\pi^N(X,x)$ and the unipotent fundamental group scheme $\pi^{uni}(X,x)$. Later, Mehta and Subramanian \cite{MeSu08} defined the local fundamental group scheme $\pi^{Loc}(X,x)$ in positive characteristic, while Amrutiya and Biswas \cite{AmBi10} introduced the F-fundamental group scheme $\pi^F(X,x)$. Langer \cite{Lan11} systematically studied the S-fundamental group scheme $\pi^S(X,x)$, and Otabe \cite{Ota17} introduced the EN-fundamental group scheme $\pi^{EN}(X,x)$. The study of Lefschetz type theorems for these fundamental group schemes was initiated by Biswas and Holla \cite{BiHo07}, who showed that if $X$ is a smooth projective variety of dimension $d \ge 3$ over an algebraically closed field $k$ and $L$ is an ample line bundle over X, then there exists integer $d_0>0$ such that for any $d>d_0$ and any smooth divisor $D\in |L^{\otimes d}|$, the natural homomorphism $\pi^{N}(D,x)\rightarrow \pi^{N}(X,x)$ is an isomorphism. Langer \cite{Lan11} showed that if $X$ is a smooth projective variety of dimension $d \ge 3$ over an algebraically closed field $k$, and $D \subsetneq X$ is a sufficiently positive smooth ample divisor, then the natural homomorphisms
\[
\pi^S(D,x)\to\pi^S(X,x), \qquad
\pi^N(D,x)\to\pi^N(X,x), \qquad
\pi^{\acute{e}t}(D,x)\to\pi^{\acute{e}t}(X,x)
\]
are isomorphisms. 

More generally, Lefschetz type problems for fundamental group schemes ask how much of the global Tannakian information of a variety can be recovered from an ample divisor. This question is especially subtle in positive characteristic. In that setting, there are several notions of fundamental group schemes, and vector bundles often behave in a delicate way under restriction. A general framework is therefore desirable. It can help clarify the relations among fundamental group schemes and reveal the categorical ideas behind geometric arguments. From this point of view, the main issues are the behavior of restriction functors, the preservation of irreducibility, and extension properties in suitable Tannakian categories.

The aim of this paper is to develop a general and purely Tannakian framework for Lefschetz type theorems. This framework not only recovers known results for several specific fundamental group schemes, but also extends them to perfect fields and to a wider class of fundamental group schemes. Our method is based on a careful study of the restriction functor to an ample divisor and its relation to the saturation of Tannakian categories. In particular, we give necessary and sufficient conditions for the induced homomorphism of Tannaka group schemes to be faithfully flat, a closed immersion, or an isomorphism.

By applying the general Tannakian framework to the specific categories of vector bundles that define various fundamental group schemes, we obtain the following Lefschetz type isomorphisms, generalizing the results of Langer \cite{Lan11} to perfect fields and extending them to new fundamental group schemes.

\begin{Theorem}[Proposition~\ref{App0} \& Proposition~\ref{app+} \& Proposition~\ref{applift}]
Let $k$ be a perfect field, $X$ a smooth variety proper over $k$ of dimension $d$, $D\subsetneq X$ a smooth ample effective divisor, $x\in D(k)$. Then 
\begin{enumerate}
    \item Let $\chara k=0$, $X$ be a smooth projective variety, $*\in\{S,\acute{e}t,E\acute{e}t,uni\}$. Then
    \begin{enumerate}
        \item If $d\geq 2$, then the induced homomorphism $\pi^*(D,x)\rightarrow \pi^{*}(X,x)$ is faithfully flat.
        \item If $d\geq 3$, then the induced homomorphism $\pi^*(D,x)\rightarrow \pi^*(X,x)$ is an isomorphism.
    \end{enumerate}
    \item Let $\chara k=p>0$, $X$ be a smooth projective variety over $k$ of dimension $d\geq 3$, $H$ an ample divisor on $X$, $\alpha$ a nonnegative integer such that $T_X(\alpha H)$ is globally generated, $D-\alpha H$ ample and $DH^{d-1}> \max\{p\alpha H^d,(d+1)\alpha H^d-K_X H^{d-1}\}$, $*\in\{S,N,EN,F, EF,Loc,ELoc,\acute{e}t,E\acute{e}t,uni\}$. Then the induced homomorphism $\pi^*(D,x)\rightarrow \pi^*(X,x)$ is an isomorphism.
    \item Let $\chara k=p>0$, $X$ be liftable to $W_2(k)$, $*\in\{N,EN,F, EF,Loc,ELoc,\acute{e}t,E\acute{e}t,uni\}$. Then
    \begin{enumerate}
        \item If $d\geq 2$, then the induced homomorphism $\pi^*(D,x)\rightarrow \pi^*(X,x)$ is faithfully flat.
        \item If $d\geq 3$ and $p\geq 3$, then the induced homomorphism $\pi^*(D,x)\rightarrow \pi^*(X,x)$ is an isomorphism.
    \end{enumerate}
\end{enumerate}
\end{Theorem}

\section{Preliminaries}

Let $k$ be a field, $K/k$ a field extension, $f:X\rightarrow S$ a morphism of schemes over $k$, $s\in S(k)$, $G$ an affine group scheme over $k$, $\Qcoh(X)$ the category of quasi-coherent sheaves on $X$, $\Vect(X)$ the category of vector bundles on $X$. Consider the following Cartesian diagrams
\[
    \begin{aligned}
        \begin{tikzcd}
     X_s\arrow{r}{t} \arrow{d}{g} & X \arrow{d}{f}\\
     \Spec(\kappa(s))\arrow{r}{s} & S
\end{tikzcd}&\quad\quad&\begin{tikzcd}
            X\times_{\Spec k}\Spec K\arrow[r,"p"]\arrow[d]& X\arrow[d]\\
            \Spec K\arrow[r]&\Spec k
        \end{tikzcd}&\quad\quad&\begin{tikzcd}
            G\times_{\Spec k}\Spec K\arrow[r,"p"]\arrow[d]& G\arrow[d]\\
            \Spec K\arrow[r]&\Spec k
        \end{tikzcd}
    \end{aligned}
\]
We denote $X_K := X \times_{\Spec k} \Spec K$, $G_K := G \times_{\Spec k}\Spec K$ and $x_K \in X_K(K)$ a point lying over $x\in X(k)$. For a coherent sheaf $E$ on $X$, we denote its pullback $p^*E$ simply by $E \otimes_k K$. 

\begin{Lemma}[{\cite[Chapter~VII, Proposition~2.3 \& 7.2]{Mil12}}]\label{isofield}
    Let $k$ be a field, $K/k$ a field extension, $f:G\rightarrow H$ a homomorphism of affine group schemes over $k$. Then $f$ is an isomorphism $($resp. faithfully flat or a closed immersion$)$ iff $f_K:G_K\rightarrow H_K$ is an isomorphism $($resp. faithfully flat or a closed immersion$)$.
\end{Lemma}

\begin{Theorem}[{\cite[Theorem A.1]{EHS07}}]\label{Thm1}
Let $L\xrightarrow{q}G\xrightarrow{p} A$ be a sequence of homomorphisms of affine group schemes over a field $k$. It induces a sequence of functors:
$$\Rep_k^f(A)\xrightarrow{p^*}\Rep_k^f(G)\xrightarrow{q^*}\Rep_k^f(L),$$
where $\Rep_k^f$ denotes the category of finite dimensional representations over $k$. Then we have the following:
\begin{enumerate}
    \item[(1)] The group homomorphism $p: G\rightarrow A$ is surjective $($faithfully flat$)$ iff  $p^*\Rep_k^f(A)$ is a full subcategory of $\Rep_k^f(G)$ and closed under taking subobjects $($subquotients$)$.
    \item[(2)] The group homomorphism $q:L\rightarrow G$ is injective $($a closed immersion$)$ iff any object of $\Rep_k^f(L)$ is a subquotient of an object of the form $q^*(V)$ for some $V\in \Rep_k^f(G)$.
    \item[(3)] Suppose $p$ is faithfully flat. Then the sequence $L\xrightarrow{q}G\xrightarrow{p} A$ is exact iff the following conditions are fulfilled:
    \begin{enumerate}
        \item[(a)] For an object $V\in \Rep_k^f(G)$,  $q^* V\in\Rep_k^f(L)$ is trivial iff $V\cong p^*U$ for some $U\in\Rep_k^f(A)$.\label{3a}
        \item[(b)]\label{3b} Let $W_0$ be the maximal trivial subobject of $q^*V$ in $\Rep_k^f(L)$. Then there exists $V_0\hookrightarrow V$ in $\Rep_k^f(G)$ such that $q^*V_0\cong W_0$.
        \item[(c)] For any $W\in\Rep_k^f(G)$ and any quotient $q^*W\twoheadrightarrow W'\in\Rep_k^f(L)$, there exists $V\in \Rep_k^f(G)$  and an embedding $W'\hookrightarrow q^*V$.
    \end{enumerate}
\end{enumerate}
\end{Theorem}

\begin{Definition}
    Let $k$ be a field, $X$ a connected scheme proper over $k$, $x\in X(k)$, $\mathcal{C}_X$ a Tannakian category over $X$ whose objects consist of vector bundles on $X$ with fibre functor $|_x:E\mapsto E|_x$, and we denote its Tannaka group scheme by $\pi(\mathcal{C}_X,x)$.
\end{Definition}

Let $k$ be a field, $X$ a connected scheme proper over $k$, $x\in X(k)$, $\mathcal{C}_X$ a Tannakian category whose objects consist of vector bundles on $X$ with fibre functor $|_x$, and $\pi(\mathcal{C}_X,x)$ its corresponding Tannaka group scheme. For any subset $M$ of $\mathcal{C}_X$, we denote $$M_1:=\{E~|~E\in M\text{ or }E^\vee\in M\}\text{, }M_2:=\{~E_1\otimes E_2\otimes\cdots \otimes E_m~|~E_i\in M_1, 1\leq i\leq m, m\in\mathbb{N}\},$$
$$\langle M\rangle_{\mathcal{C}_X}:=
\left\{ 
    E\in \mathcal{C}_X \left\vert 
    \begin{split}
        & \exists F_i\in M_2,1\leq i\leq t, \text{ and } E_1,E_2\in \mathcal{C}_X\\
        & s.t. \text{ }E_1\hookrightarrow E_2\hookrightarrow\bigoplus^{t}_{i=1}F_i,\text{ and } E\cong E_2/E_1
    \end{split}\right.
\right\}.
$$
Then $(\langle M\rangle_{\mathcal{C}_X},\otimes,|_x,\mathcal{O}_X)$ is a Tannakian subcategory of $(\mathcal{C}_X,\otimes,|_x,\mathcal{O}_X)$ and we denote its Tannaka group scheme by $\pi(\langle M\rangle_{\mathcal{C}_X},x)$.

\begin{Definition}
    Let $k$ be a field, $X$ a connected scheme proper over $k$, $\mathcal{V}_X$ a full subcategory of $\Vect(X)$. Define the \textit{saturation category} $\overline{\mathcal{V}}_X$ of $\mathcal{V}_X$ as the full subcategory of $\Vect(X)$ whose objects are those $E$ for which there exists a filtration
    $$0\hookrightarrow E_1 \hookrightarrow \cdots\hookrightarrow E_n=E,$$
    such that $E^i=E_{i+1}/E_i\in\mathcal{V}_X$ for any $i$.
\end{Definition}

\begin{Remark}
    Let $k$ be a field, $X$ a connected scheme proper over $k$, $x\in X(k)$, $\mathcal{V}_X$ a full subcategory of $\Vect(X)$. If $\mathcal{V}_X$ is a rigid (resp. tensor) subcategory of $\Vect(X)$, then $\overline{\mathcal{V}}_X$ is also a rigid (resp. tensor) subcategory of $\Vect(X)$. In particular, if $\mathcal{C}_X$ is a neutral Tannakian category over $X$ with fibre functor $|_x$, then by \cite[Proposition~3.3]{LiTiBa26}, $\overline{\mathcal{C}}_X$ is a neutral Tannakian category over $X$ with fibre functor $|_x$, and $\mathcal{C}_X$ is a Tannakian subcategory of $\overline{\mathcal{C}}_X$, which induces a faithfully flat homomorphism $\pi(\overline{\mathcal{C}}_X,x)\twoheadrightarrow \pi(\mathcal{C}_X,x)$.
\end{Remark}

\begin{Remark}
    Let $k$ be a field, $f:X\rightarrow S$ a morphism of connected schemes proper over $k$, $x\in X(k)$ lying over $s\in S(k)$, $\mathcal{C}_X,\mathcal{C}_S$ the Tannakian categories over $X$, $S$ respectively. If pullback induces a functor $f^*:\mathcal{C}_S\rightarrow \mathcal{C}_X$, then pullback also induces a functor $f^*:\overline{\mathcal{C}}_S\rightarrow \overline{\mathcal{C}}_X$ by \cite[Lemma~4.5]{LiTiKu26}. 
\end{Remark}

\begin{Definition}
    Let $X$ be a locally noetherian scheme, $E$ a coherent sheaf on $X$, $n$ a nonnegative integer. $X$ is said to satisfy \textit{Serre's condition} $S_n$, if 
    $\depth(\mathcal{O}_{X,x})\geq \min\{n,\dim(\mathcal{O}_{X,x})\}$
    for any $x\in X$. $E$ is said to satisfy \textit{Serre's condition} $S_n$, if 
    $\depth(E_{x})\geq \min\{n,\dim(\mathcal{O}_{X,x})\}$
    for any $x\in X$. For simplicity, we say $X$ satisfies $S_n$ and $E$ satisfies $S_n$.
\end{Definition}

\begin{Definition}[{\cite[p.164]{Har70}}]
    Let $X$ be a scheme, $D\subsetneq X$ a closed subscheme. Define $\GLF(D)$ to be the category of germs of locally free sheaves around $D$. The objects of $\GLF(D)$ are equivalence classes of pairs $(E_U,U)$ where $U\subseteq X$ is an open subset with $D\subsetneq U$ and $E_U\in\Vect(U)$. Two pairs $(E_U,U)$ and $(E_V,V)$ are equivalent if there exists open subset $W\subseteq U\cap V$ with $D\subsetneq W$ such that $E_U|_W\cong E_V|_W$. The morphism between two objects $(E_U,U)$ and $(E_V,V)$ is a morphism of locally free sheaves between $E_U|_{U\cap V}$ and $E_V|_{U\cap V}$.
\end{Definition}

\begin{Lemma}[cf. {\cite[Proposition~1.11]{Har94}}]\label{pushforwardopenimmersion}
    Let $X$ be a noetherian scheme, $U\subsetneq X$ an open subset with complement of codimension $\geq 2$, $i:U\rightarrow X$ the open immersion, $E$ a coherent sheaf on $X$ satisfying $S_2$. Then the natural map $E\rightarrow i_*(E|_U)$ is an isomorphism.
\end{Lemma}

\begin{Lemma}
    Let $k$ be a field, $X$ a scheme over $k$ of dimension $\geq 2$ satisfying $S_2$, $D\subsetneq X$ an ample effective divisor. Then there exists a fully faithful functor $\GLF(D)\rightarrow \Coh(X), (E_U,U)\mapsto {i_U}_*E_U$ where $i_U:U\rightarrow X$ is the open immersion.
\end{Lemma}

\begin{proof}
    Let $(E_U,U), (E_V,V)\in\GLF(D)$ such that $V\subseteq U$ and $E_U|_V\cong E_V$. We have the following commutative diagram
    \[
        \begin{tikzcd}
            V\arrow[r,hook,"i_{UV}"]\arrow[rr,hook, bend right,"i_{V}"]&U\arrow[r,hook,"i_{U}"]&X
        \end{tikzcd}
    \]
    where $i_{UV}, i_{U}, i_{V}$ are open immersions. Since $X$ satisfies $S_2$ and $\codim(X\setminus V,X)\geq 2$, we have $U$ satisfies $S_2$ and $\codim(U\setminus V,U)\geq 2$. It follows that ${i_{UV}}_*(E_U|_V)\cong E_U$ by Lemma~\ref{pushforwardopenimmersion}. Then we have
    $$\begin{aligned}
        {i_V}_*E_V={i_{U}}_*{i_{UV}}_*E_V\cong {i_{U}}_*{i_{UV}}_*(E_U|_V)\cong {i_{U}}_*E_U.
    \end{aligned}$$
    Hence the functor is well defined.

    Let $(E_U,U), (E'_U,U)\in\GLF(D)$. Then we have 
    $$\begin{aligned}
        \Hom_{X}({i_U}_*E_U,{i_U}_*E'_U)\cong \Hom_{U}(({i_U}_*E_U)|_U,E'_U)\cong \Hom_{U}(E_U,E'_U)
    \end{aligned}.$$
    Hence the functor is fully faithful.
\end{proof}

\section{The Lefschetz type theorem for Tannaka group schemes}

\textbf{Notation and Conventions:}
Let $k$ be a field, $f:X\rightarrow S$ a morphism of connected schemes proper over $k$, $x\in X(k)$ lying over $s\in S(k)$, $\mathcal{C}_X$ $(\text{resp. }\mathcal{C}_S)$ Tannakian category over $X$ $(\text{resp. }S)$ respectively whose objects consist of vector bundles on $X$ $(\text{resp. }S)$ with fibre functor $|_x$ $(\text{resp. }|_s)$, $\pi(\mathcal{C}_X,x),\pi(\mathcal{C}_S,s)$ the corresponding Tannaka group schemes. The pullback $f^*$ induces a functor $f^*:\mathcal{C}_S\rightarrow \mathcal{C}_X$ and a canonical homomorphism $\pi(\mathcal{C}_X,x)\rightarrow \pi(\mathcal{C}_S,s)$ of group schemes.

\begin{Proposition}\label{generalsur}
Let $k$ be a field, $\mathcal{C},\mathcal{D}$ neutral Tannakian categories over $k$ with fibre functor $\omega$, $\pi_C$, $\pi_D$ their Tannaka group schemes respectively, $\Phi:\mathcal{C}\rightarrow \mathcal{D}$ an exact tensor functor, $\varphi:\pi_D\rightarrow \pi_C$ its induced homomorphism. Then the following conditions are equivalent:
\begin{enumerate}
    \item The induced homomorphism $\varphi:\pi_D\rightarrow \pi_C$ is faithfully flat.
    \item \begin{enumerate}
    \item The functor $\Phi$ is fully faithful.
    \item For any irreducible object $E\in \mathcal{C}$, $\Phi (E)$ is irreducible in $\mathcal{D}$.
\end{enumerate}
\end{enumerate}
\end{Proposition}

\begin{proof}
$(1)\Rightarrow (2)$ Suppose there exists irreducible object $E_0\in\mathcal{C}$ such that $\Phi(E_0)$ has a proper subobject $F_0\hookrightarrow \Phi(E_0)\in\mathcal{D}$. Then there exists $E'\hookrightarrow E_0$ such that $F_0\cong \Phi(E')$ by [Theorem~\ref{Thm1}, (1)]. Since $E_0$ is irreducible, we have $E'=0$ or $E'= E_0$. Each case yields a contradiction since $F_0\cong \Phi(E')\hookrightarrow \Phi(E_0)\in\mathcal{D}$ is a proper subobject.

$(2)\Rightarrow (1)$ Let $E\in\mathcal{C}$ and $F\hookrightarrow \Phi(E)\in\mathcal{D}$. Then there exists a Jordan H\"older filtration of $\Phi(E)$ in $\mathcal{D}$:
$$0=F_0\hookrightarrow F_1\hookrightarrow \cdots \hookrightarrow F_m=\Phi(E)$$
such that $F=F_{m_0}$ for some integer $m_0$. Consider the Jordan H\"older filtration of $E$ in $\mathcal{C}$:
$$0=E_0\hookrightarrow E_1\hookrightarrow \cdots \hookrightarrow E_n=E.$$
Applying the exact functor $\Phi$, we obtain a filtration of $\Phi(E)$ in $\mathcal{D}$:
$$0=\Phi(E_0)\hookrightarrow \Phi(E_1)\hookrightarrow \cdots \hookrightarrow \Phi(E_n)=\Phi(E).$$
Since $E_{i+1}/E_i$ is irreducible in $\mathcal{C}$, then $\Phi(E_{i+1})/\Phi(E_i)\cong \Phi(E_{i+1}/E_i)$ is irreducible in $\mathcal{D}$. It follows that 
$$0=\Phi(E_0)\hookrightarrow \Phi(E_1)\hookrightarrow \cdots \hookrightarrow \Phi(E_n)=\Phi(E)$$
is a Jordan H\"older filtration of $\Phi(E)$ in $\mathcal{D}$ and $n=m$. So for any $F_{i+1}/F_i$, there exists $E_{j_i+1}/E_{j_i}$ with 
$F_{i+1}/F_i\cong \Phi(E_{j_i+1}/E_{j_i})$.
Then there exists $j_0$ with $F_1\cong \Phi(E_{j_0+1}/E_{j_0})$. Since $\Phi$ is fully faithful, we have
$$\Hom_{\mathcal{C}}(E_{j_0+1}/E_{j_0},E)=\Hom_{\mathcal{D}}(F_1,\Phi(E)).$$
So there exists a morphism $f:E_{j_0+1}/E_{j_0}\rightarrow E$ whose image under $\Phi$ is the embedding $\Phi (f):F_1\hookrightarrow \Phi(E)$. It follows that $f$ is an embedding. Without loss of generality, we assume $j_0=0$, so that $\Phi(E_1)=F_1$. Consider the Jordan H\"older filtration of $E/E_1$ in $\mathcal{C}$:
$$0=E_1/E_1\hookrightarrow E_2/E_1\hookrightarrow \cdots \hookrightarrow E_n/E_1=E/E_1$$
and the Jordan H\"older filtration of $\Phi(E)/F_1$ in $\mathcal{D}$:
$$0=F_1/F_1\hookrightarrow F_2/F_1\hookrightarrow \cdots \hookrightarrow F_{n}/F_1=\Phi(E/E_1)$$
Applying the same method to $F_2/F_1$, we have $F_2/F_1\cong \Phi (E_{j_1+1}/E_{j_1})$ for some $1\leq j_1\leq n-1$ and $E_{j_1+1}/E_{j_1}\hookrightarrow E/E_1$. It follows that $F_2\cong \Phi(E_2')$ for some $E_2'\hookrightarrow E$. Repeating this method, we have $F=F_{n_0}\cong \Phi( E_{n_0}')$ for some $E_{n_0}'\hookrightarrow E$. Then $\Phi$ is closed under taking subobjects. Moreover, $\Phi$ is fully faithful, hence the induced homomorphism $\varphi:\pi_D\rightarrow \pi_C$ is faithfully flat by [Theorem~\ref{Thm1}, (1)].
\end{proof}

\begin{Proposition}\label{H01fullyfaithful}
    Let $k$ be a field, $X$ a connected scheme proper over $k$, $D\subsetneq X$ an effective connected divisor, $\mathcal{V}_X$ a rigid tensor full subcategory of $\Vect(X)$. Then we have 
    \begin{enumerate}
        \item $H^0(X,E(-D))=0$ for any $E\in\mathcal{V}_X$ iff the functor $|_D:\mathcal{V}_X\rightarrow \Vect(D)$ is faithful.
        \item If $H^i(X,E(-D))=0$ for any $E\in\mathcal{V}_X$ $(i=0,1)$, then the functor $|_D:\mathcal{V}_X\rightarrow \Vect(D)$ is fully faithful.
        \item If $\mathcal{V}_X$ is a saturated category, then $H^i(X,E(-D))=0$ for any $E\in\mathcal{V}_X$ $(i=0,1)$ iff the functor $|_D:\mathcal{V}_X\rightarrow \Vect(D)$ is fully faithful.
    \end{enumerate}
\end{Proposition}

\begin{proof}
    For any $E\in\mathcal{V}_X$, we have the following exact sequence
    $$0\rightarrow E(-D)\rightarrow E\rightarrow E|_D\rightarrow 0.$$
    Then we have the following long exact sequence
    $$0\rightarrow H^0(X,E(-D))\rightarrow H^0(X,E)\rightarrow H^0(D,E|_D)\rightarrow H^1(X,E(-D))\rightarrow H^1(X,E)\rightarrow H^1(D,E|_D)\rightarrow$$
    $$\quad\quad\quad\quad\quad\quad\quad\quad\quad\quad\quad H^2(X,E(-D))\rightarrow H^2(X,E)\rightarrow H^2(D,E|_D)\rightarrow \cdots.\quad\quad\quad\quad\quad\quad\quad\quad\quad\quad (*)$$

    (1) $(\Rightarrow)$ For any $E,E'\in\mathcal{V}_X$, we have $\mathcal{H}om_X(E,E')\in\mathcal{V}_X$, so $H^0(X,\mathcal{H}om_X(E,E')(-D))=0$. Then by $(*)$ we have an exact sequence 
    $$0\rightarrow H^0(X,\mathcal{H}om_X(E,E')(-D))\rightarrow \Hom_X(E,E')\rightarrow \Hom_D(E|_D,E'|_D).$$
    It follows that $\Hom_X(E,E')\hookrightarrow \Hom_D(E|_D,E'|_D)$, i.e. the functor $|_D:\mathcal{V}_X\rightarrow \Vect(D)$ is faithful.

    $(\Leftarrow)$ Since the functor $|_D:\mathcal{V}_X\rightarrow \Vect(D)$ is faithful, we have $H^0(X,E)\hookrightarrow H^0(D,E|_D)$ for any $E\in\mathcal{V}_X$. It follows that $H^0(X,E(-D))=0$ for any $E\in\mathcal{V}_X$ by $(*)$.

    (2) For any $E,E'\in\mathcal{V}_X$, we have $\mathcal{H}om_X(E,E')\in\mathcal{V}_X$. Then we have a long exact sequence 
    $$0\rightarrow H^0(X,\mathcal{H}om_X(E,E')(-D))\rightarrow \Hom_X(E,E')\rightarrow \Hom_D(E|_D,E'|_D)\rightarrow$$
	$$H^1(X,\mathcal{H}om_X(E,E')(-D))\rightarrow \cdots.$$
    Since $H^0(X,\mathcal{H}om_X(E,E')(-D))=H^1(X,\mathcal{H}om_X(E,E')(-D))=0$, we have $$\Hom_X(E,E')\cong \Hom_D(E|_D,E'|_D),$$
    which implies that the functor $|_D:\mathcal{V}_X\rightarrow \Vect(D)$ is fully faithful.

    (3) $(\Rightarrow)$ It follows by (2).
    
    $(\Leftarrow)$ Suppose $|_D:\mathcal{V}_X\rightarrow \Vect(D)$ is fully faithful. Then $H^0(X,E(-D))=0$ for any $E\in\mathcal{C}_X$ by (1).

    Let $E\in\mathcal{C}_X$, we have $H^1(X,E)= H^1(X,\mathcal{H}om_X(\mathcal{O}_X,E))$ and $H^1(D,E|_D)= H^1(D,\mathcal{H}om_D(\mathcal{O}_D,E|_D))$. For any $0\rightarrow E\rightarrow E_1\rightarrow \mathcal{O}_X\rightarrow 0$ and $0\rightarrow E\rightarrow E_2\rightarrow \mathcal{O}_X\rightarrow 0$ in $\Ext^1_X(\mathcal{O}_X,E)$ such that there exists a commutative diagram on $D$:
    \[
        \begin{tikzcd}
            0\arrow[r] &E|_D\arrow[d,equal] \arrow[r]&E_1|_D\arrow[d]\arrow[r]&\mathcal{O}_D\arrow[d,equal]\arrow[r]&0\\
            0\arrow[r] &E|_D\arrow[r]&E_2|_D\arrow[r]&\mathcal{O}_D\arrow[r]&0.
        \end{tikzcd}
    \]
    Since $|_D$ is fully faithful and $\mathcal{V}_X$ is a saturated category, there exists a commutative diagram in $\mathcal{V}_X$:
    \[
        \begin{tikzcd}
            0\arrow[r] &E\arrow[d,equal] \arrow[r]&E_1\arrow[d]\arrow[r]&\mathcal{O}_X\arrow[d,equal]\arrow[r]&0\\
            0\arrow[r] &E\arrow[r]&E_2\arrow[r]&\mathcal{O}_X\arrow[r]&0.
        \end{tikzcd}
    \]
    It follows that $0\rightarrow E\rightarrow E_1\rightarrow \mathcal{O}_X\rightarrow 0$ and $0\rightarrow E\rightarrow E_2\rightarrow \mathcal{O}_X\rightarrow 0$ are equivalent in $\Ext^1_X(\mathcal{O}_X,E)$. Then we have $H^1(X,E)=\Ext^1_X(\mathcal{O}_X,E)\hookrightarrow \Ext^1_D(\mathcal{O}_D,E|_D)=H^1(D,E|_D)$. Since the functor $|_D$ is fully faithful, we also have $H^0(X,E)\xrightarrow{\cong} H^0(D,E|_D)$. Hence $H^1(X,E(-D))=0$ for any $E\in\mathcal{V}_X$ by $(*)$.
\end{proof}

\begin{Proposition}\label{Hisaturated}
    Let $k$ be a field, $X$ a connected scheme proper over $k$, $D\subsetneq X$ an effective connected divisor, $\mathcal{V}_X$ a rigid tensor full subcategory of $\Vect(X)$, $i\in\mathbb{N}$. If $H^i(X,E(-D))=0$ for any $E\in \mathcal{V}_X$, then $H^i(X,E'(-D))=0$ for any $E'\in \overline{\mathcal{V}}_X$.
\end{Proposition}

\begin{proof}
    For any $E'\in\overline{\mathcal{V}}_X$, consider the filtration of $E'$ in $\overline{\mathcal{V}}_X$:
    $$0=E_0\hookrightarrow E_1\hookrightarrow \cdots\hookrightarrow E_n=E',$$
    where $E_{j+1}/E_{j}\in\mathcal{V}_X$ for any $0\leq j\leq n-1$. Consider the exact sequence
    $$0\rightarrow E_1\rightarrow E_2\rightarrow E_2/E_1\rightarrow 0.$$
    Tensoring with $\mathcal{O}_X(-D)$, we have a long exact sequence
    $$\begin{aligned}
        0\rightarrow H^0(X,E_1(-D))\rightarrow H^0(X,E_2(-D))\rightarrow H^0(X,E_2/E_1(-D))\rightarrow \cdots\\
        H^i(X,E_1(-D))\rightarrow H^i(X,E_2(-D))\rightarrow H^i(X,E_2/E_1(-D))\rightarrow \cdots.
    \end{aligned}$$
    It follows that $H^i(X,E_2(-D))=0$. Repeating this method, we have $H^i(X,E_j(-D))=0$ for any $1\leq j\leq n$. In particular, we have $H^i(X,E'(-D))=0$.
\end{proof}

\begin{Corollary}\label{H01generalsur}
    Let $k$ be a field, $X$ a connected scheme proper over $k$, $D\subsetneq X$ an effective connected divisor, $x\in D(k)$. Then the following conditions are equivalent:
\begin{enumerate}
    \item \begin{enumerate}
        \item $H^i(X,E(-D))=0$ for any $E\in\mathcal{C}_X$, $i=0,1$.
        \item For any irreducible object $E\in \mathcal{C}_X$, $E|_D$ is irreducible in $\mathcal{C}_D$.
        \end{enumerate}
    \item \begin{enumerate}
        \item $H^i(X,E'(-D))=0$ for any $E'\in\overline{\mathcal{C}}_X$, $i=0,1$.
        \item For any irreducible object $E'\in \overline{\mathcal{C}}_X$, $E'|_D$ is irreducible in $\overline{\mathcal{C}}_D$.
        \end{enumerate}
    \item The induced homomorphism $\pi(\overline{\mathcal{C}}_D,x)\rightarrow \pi(\overline{\mathcal{C}}_X,x)$ is faithfully flat.
\end{enumerate}
If the above conditions hold, then the induced homomorphism $\pi(\mathcal{C}_D,x)\rightarrow \pi(\mathcal{C}_X,x)$ is faithfully flat.
\end{Corollary}

\begin{proof}
    $(1)\Rightarrow (2)$ Since $\mathcal{C}_X$ (resp. $\mathcal{C}_D$) has the same irreducible objects as $\overline{\mathcal{C}}_X$ (resp. $\overline{\mathcal{C}}_D$), it follows by Proposition~\ref{Hisaturated}.

    $(2)\Rightarrow (3)$ It follows by Proposition~\ref{generalsur} and [Proposition~\ref{H01fullyfaithful}, (2)].

    $(3)\Rightarrow (1)$ It follows by Proposition~\ref{generalsur} and [Proposition~\ref{H01fullyfaithful}, (3)].

    Suppose (3) holds and we have the following commutative diagram
    \[
      \begin{tikzcd}
          \pi(\overline{\mathcal{C}}_D,x)\arrow[r,two heads]\arrow[d,two heads] &\pi(\overline{\mathcal{C}}_X,x)\arrow[d,two heads]\\
          \pi(\mathcal{C}_D,x)\arrow[r,two heads]& \pi({\mathcal{C}}_X,x)
      \end{tikzcd}  
    \]
    It follows that the induced homomorphism $\pi(\mathcal{C}_D,x)\rightarrow \pi(\mathcal{C}_X,x)$ is faithfully flat.
\end{proof}

\begin{Proposition}\label{Lefschetzclosedimmersion}
    Let $k$ be a field, $X$ a connected scheme proper over $k$ of dimension $d\geq 2$ satisfying $S_2$, $D\subsetneq X$ an ample effective connected divisor, $x\in D(k)$. Then the following conditions are equivalent:
    \begin{enumerate}
    \item The induced homomorphism $\pi(\mathcal{C}_D,x)\rightarrow \pi(\mathcal{C}_X,x)$ is a closed immersion.
    \item For any $E_D\in \mathcal{C}_D$, there exists $(E_U,U)\in\GLF(D)$ such that
    \begin{enumerate}
    \item $E_U|_D\in\mathcal{C}_D$ and $E_D$ is a subquotient of $E_U|_D\in\mathcal{C}_D$ in $\mathcal{C}_D$.
    \item ${j_U}_*E_U\in\mathcal{C}_X$ where $j_U:U\hookrightarrow X$ is the open immersion.
\end{enumerate}
\end{enumerate}
\end{Proposition}

\begin{proof}
    $(1)\Rightarrow (2)$ For any $E_D\in\mathcal{C}_D$, there exists $E\in\mathcal{C}_X$ such that $E_D$ is a subquotient of $E|_D$ in $\mathcal{C}_D$ by [Theorem~\ref{Thm1}, (2)]. Take $(E,X)\in \GLF(D)$, then it follows immediately.

    $(2)\Rightarrow (1)$ For any $E_D\in\mathcal{C}_D$, there exists $(E_U,U)\in\GLF(D)$ such that $E_U|_D\in\mathcal{C}_D$ and $E_D$ is a subquotient of $E_U|_D\in\mathcal{C}_D$ in $\mathcal{C}_D$, and ${j_U}_*E_U\in\mathcal{C}_X$ where $j_U:U\hookrightarrow X$ is the open immersion. Note that ${j_U}_*E_U|_D=({j_U}_*E_U|_U)|_D\cong E_U|_D$ and $E_D$ is a subquotient of $E_U|_D\in\mathcal{C}_D$ in $\mathcal{C}_D$. Hence the induced homomorphism $\varphi:\pi_D\rightarrow \pi_C$ is a closed immersion by [Theorem~\ref{Thm1}, (2)].
\end{proof}

\begin{Corollary}\label{Main?}
    Let $k$ be a field, $X$ a connected scheme proper over $k$ of dimension $d\geq 2$ satisfying $S_2$, $D\subsetneq X$ an ample effective connected divisor, $x\in D(k)$. Then the following conditions are equivalent:
    \begin{enumerate}
    \item The induced homomorphism $\pi(\mathcal{C}_D,x)\rightarrow\pi(\mathcal{C}_X,x)$ is an isomorphism.
    \item \begin{enumerate}
    	\item The functor $|_D:\mathcal{C}_X\rightarrow \mathcal{C}_D$ is fully faithful.
            \item For any irreducible object $E\in\mathcal{C}_X$, $E|_D$ is irreducible in $\mathcal{C}_D$.
            \item For any $E_D\in \mathcal{C}_D$, there exists $(E_U,U)\in\GLF(D)$ such that
    \begin{enumerate}
    \item $E_U|_D\in\mathcal{C}_D$ and $E_D$ is a subquotient of $E_U|_D\in\mathcal{C}_D$ in $\mathcal{C}_D$.
    \item ${j_U}_*E_U\in\mathcal{C}_X$ where $j_U:U\hookrightarrow X$ is the open immersion.
\end{enumerate}
            \end{enumerate}
\end{enumerate}
\end{Corollary}

\begin{proof}
    It follows immediately by Proposition~\ref{generalsur} and Proposition~\ref{Lefschetzclosedimmersion}.
\end{proof}

\begin{Proposition}\label{isosubcate}
     Let $k$ be a field, $\mathcal{C}'$ $(\text{resp. } \mathcal{D}')$ neutral Tannakian category over $k$ with fibre functor $\omega$, $\mathcal{C}\subseteq\mathcal{C}'$ $(\text{resp. }\mathcal{D}\subseteq\mathcal{D}')$ Tannakian subcategories, $\pi_C$, $\pi_{C'}$, $\pi_D$, $\pi_{D'}$ their Tannaka group schemes respectively, $\Phi':\mathcal{C}'\rightarrow \mathcal{D}'$ an exact tensor functor such that $\Phi'(E)\in\mathcal{D}$ for any $E\in\mathcal{C}$. If the induced homomorphism $\pi_{D'}\rightarrow\pi_{C'}$ is an isomorphism, then the following conditions are equivalent:
    \begin{enumerate}
        \item The induced homomorphism $\pi_{D}\rightarrow\pi_{C}$ is a closed immersion.
        \item The induced homomorphism $\pi_{D}\rightarrow\pi_{C}$ is an isomorphism.
        \item For any $E'\in\mathcal{C}'$, if $\Phi'(E')\in\mathcal{D}$ then $E'\in\mathcal{C}$.
    \end{enumerate}
\end{Proposition}

\begin{proof}
    $(1)\Rightarrow (2)$ By following commutative diagrams, the homomorphism $\pi_{D}\rightarrow\pi_{C}$ is an isomorphism.
    \[
        \begin{aligned}
        \begin{tikzcd}
           \mathcal{C}'\arrow[r,"{\Phi',\cong}"]&\mathcal{D}'\\
            \mathcal{C}\arrow[r,"{\Phi'}"]\arrow[u, hook]&\mathcal{D}\arrow[u, hook]
        \end{tikzcd}
        \quad\quad
	\begin{tikzcd}
            \pi_{D'}\arrow[r,"\cong"]\arrow[d, two heads]&\pi_{C'}\arrow[d, two heads]\\
            \pi_D\arrow[r,two heads]&\pi_C
        \end{tikzcd}
        \end{aligned}
    \]
    
    $(2)\Rightarrow (3)$ Suppose there exists $E'\in\mathcal{C}'\setminus\mathcal{C}$ such that $\Phi'(E')\in\mathcal{D}$. It follows that $\mathcal{C}\subsetneq \langle \mathcal{C}\cup \{E'\}\rangle_{\mathcal{C}'}$ and $\langle \{\Phi'(E)\}_{E\in\mathcal{C}\cup \{E'\}}\rangle_{\mathcal{D}'}=\mathcal{D}$.
    Then the natural homomorphism $\pi_{\langle \mathcal{C}\cup \{E'\}\rangle_{\mathcal{C}'}}\twoheadrightarrow\pi_{\mathcal{C}}$ is not an isomorphism and $\pi_{\langle \{\Phi'(E)\}_{E\in\mathcal{C}\cup \{E'\}}\rangle_{\mathcal{D}'}}\cong \pi_{\mathcal{D}}$.
    Then we have the following commutative diagram
    \[
        \begin{tikzcd}
             \pi_{D'}\arrow[r, two heads]\arrow[d,"\cong"]&\pi_{\langle \{\Phi'(E)\}_{E\in\mathcal{C}\cup \{E'\}}\rangle_{\mathcal{D}'}}\arrow[d,two heads]\arrow[r,"\cong"]&\pi_{D}\arrow[d,"\cong"]\\
            \pi_{C'}\arrow[r, two heads]&\pi_{\langle \mathcal{C}\cup \{E'\}\rangle_{\mathcal{C}'}}\arrow[r, two heads, "\not\cong"]&\pi_{C}
        \end{tikzcd}
    \]
    This commutative diagram yields a contradiction. Hence for any $E'\in\mathcal{C}'$, if $\Phi'(E')\in\mathcal{D}$ then $E'\in\mathcal{C}$.

    $(3)\Rightarrow (1)$ For any $F\in \mathcal{D}\subseteq \mathcal{D}'$, there exists $E'\in\mathcal{C}'$ such that $\Phi'(E')= F\in\mathcal{D}$. It follows that $E'\in\mathcal{C}$. In other words, for any $F\in \mathcal{D}$, there exists $E'\in\mathcal{C}$ such that $\Phi'(E')=F$.  Hence the induced homomorphism $\pi_{D}\rightarrow\pi_{C}$ is a closed immersion by [Theorem~\ref{Thm1}, (2)].
\end{proof}

\begin{Proposition}\label{isoSerre}
    Let $k$ be a field, $f:X\rightarrow S$ a morphism of connected schemes proper over $k$, $x\in X(k)$ lying over $s\in S(k)$, $\mathcal{C}'_X,\mathcal{C}'_S$ the Tannakian categories over $X$, $S$ respectively, $\mathcal{C}_X\subseteq \mathcal{C}'_X,\mathcal{C}_S\subseteq \mathcal{C}'_S$ the Tannakian subcategories, pullback induces functors $f^*:\mathcal{C}'_S\rightarrow \mathcal{C}'_X$ and $f^*:\mathcal{C}_S\rightarrow \mathcal{C}_X$. Suppose the saturation category $\overline{\mathcal{C}}_X$ $(\text{resp. } \overline{\mathcal{C}}_S)$ is a Tannakian subcategory of $\mathcal{C}'_X$ $(\text{resp. } \mathcal{C}'_S)$. If the induced homomorphisms $\pi(\mathcal{C}'_X,x)\rightarrow \pi(\mathcal{C}'_S,s)$ and $\pi(\mathcal{C}_X,x)\rightarrow \pi(\mathcal{C}_S,s)$ are isomorphisms, then the induced homomorphism $\pi(\overline{\mathcal{C}}_X,x)\rightarrow \pi(\overline{\mathcal{C}}_S,s)$ is an isomorphism.
\end{Proposition}

\begin{proof}
    For any $E\in\overline{\mathcal{C}}_X\subseteq\mathcal{C}'_X$, since the induced homomorphism $\pi(\mathcal{C}'_X,x)\rightarrow \pi(\mathcal{C}'_S,s)$ is an isomorphism, there exists $F\in\mathcal{C}'_S$ such that $E\cong f^*F$. Consider the filtration of $E$ in $\overline{\mathcal{C}}_X$:
    $$0=E_0\hookrightarrow E_1\hookrightarrow\cdots\hookrightarrow E_n=E,$$
    where $E_{i+1}/E_i\in\mathcal{C}_X$ for any $0\leq i\leq n-1$. Since the induced homomorphism $\pi(\mathcal{C}'_X,x)\rightarrow \pi(\mathcal{C}'_S,s)$ is an isomorphism, there exists a filtration in $\mathcal{C}'_S$:
    $$0=F_0\hookrightarrow F_1\hookrightarrow\cdots\hookrightarrow F_n=F,$$
    where $E_{i}\cong f^*F_{i}$ for any $0\leq i\leq n$. Since the induced homomorphism $\pi(\mathcal{C}_X,x)\rightarrow \pi(\mathcal{C}_S,s)$ is an isomorphisms and $f^*(F_{i+1}/F_i)\cong E_{i+1}/E_i\in\mathcal{C}_X$, we have $F_{i+1}/F_i\in\mathcal{C}_S$ for any $0\leq i\leq n-1$ by Proposition~\ref{isosubcate}. It follows that $F_{i}\in\overline{\mathcal{C}}_S$ for any $1\leq i\leq n$. In particular, we have $F\in\overline{\mathcal{C}}_S$. Since the functor $f^*:\overline{\mathcal{C}}_S\rightarrow\overline{\mathcal{C}}_X$ is fully faithful, the induced homomorphism $\pi(\overline{\mathcal{C}}_X,x)\rightarrow \pi(\overline{\mathcal{C}}_S,s)$ is an isomorphism.
\end{proof}

\begin{Proposition}\label{saturated}
     Let $k$ be a field, $X$ a connected scheme proper over $k$, $D\subsetneq X$ an effective connected divisor, $x\in D(k)$. Suppose the induced homomorphism $\pi(\mathcal{C}_D,x)\rightarrow \pi(\mathcal{C}_X,x)$ is an isomorphism and $H^i(X,E(-D))=0$ for any $E\in \mathcal{C}_X$, $i=1,2$. Then the induced homomorphism $\pi(\overline{\mathcal{C}}_D,x)\rightarrow \pi(\overline{\mathcal{C}}_X,x)$ is an isomorphism.
\end{Proposition}

\begin{proof}
    By Proposition~\ref{H01fullyfaithful}, we have $H^0(X,E(-D))=0$ for any $E\in \mathcal{C}_X$. By Corollary~\ref{H01generalsur}, the induced homomorphism $\pi(\overline{\mathcal{C}}_D,x)\rightarrow \pi(\overline{\mathcal{C}}_X,x)$ is faithfully flat.
    
    For any $E_D\in \overline{\mathcal{C}}_D$, consider the filtration of $E_D$ in $\overline{\mathcal{C}}_D$:
    $$0=E_{D_0}\hookrightarrow E_{D_1}\hookrightarrow \cdots\hookrightarrow E_{D_n}=E_D,$$
    where $E_{D_{i+1}}/E_{D_i}\in\mathcal{C}_D$ for any $0\leq i\leq n-1$. Consider the exact sequence
    $$0\rightarrow E_{D_1}\rightarrow E_{D_2}\rightarrow E_{D_{2}}/E_{D_1}\rightarrow 0.$$
    Since the induced homomorphism $\pi(\mathcal{C}_D,x)\rightarrow \pi(\mathcal{C}_X,x)$ is an isomorphism, there exist $E_1, E^1\in\mathcal{C}_X$ such that $E_1|_D\cong E_{D_1}$ and $E^1|_D\cong E_{D_{2}}/E_{D_1}$. Consider the following exact sequence
    $$0\rightarrow \mathcal{H}om_X(E^1,E_1)\otimes_{\mathcal{O}_X}\mathcal{O}_X(-D)\rightarrow \mathcal{H}om_X(E^1,E_1)\rightarrow \mathcal{H}om_D(E^1|_D,E_1|_D)\rightarrow 0.$$
    Then we have a long exact sequence 
    $$\begin{aligned}
        0\rightarrow H^0(X,\mathcal{H}om_X(E^1,E_1)(-D))\rightarrow H^0(X,\mathcal{H}om_X(E^1,E_1))\rightarrow H^0(D,\mathcal{H}om_D(E^1|_D,E_1|_D))\rightarrow \\
        H^1(X,\mathcal{H}om_X(E^1,E_1)(-D))\rightarrow H^1(X,\mathcal{H}om_X(E^1,E_1))\rightarrow H^1(D,\mathcal{H}om_D(E^1|_D,E_1|_D))\rightarrow\\
        H^2(X,\mathcal{H}om_X(E^1,E_1)(-D))\rightarrow H^2(X,\mathcal{H}om_X(E^1,E_1))\rightarrow H^2(D,\mathcal{H}om_D(E^1|_D,E_1|_D))\rightarrow \cdots.
    \end{aligned}$$
    Since $\mathcal{H}om_X(E^1,E_1)\in \mathcal{C}_X$, we have $H^2(X,\mathcal{H}om_X(E^1,E_1)(-D))=0$. It follows that the natural map $H^1(X,\mathcal{H}om_X(E^1,E_1))\rightarrow H^1(X,\mathcal{H}om_D(E^1|_D,E_1|_D))$ is surjective. In other words, we obtain
    $$\Ext^1(E^1,E_1)\twoheadrightarrow \Ext^1(E^1|_D,E_1|_D).$$
    So there exists $E_2\in\overline{\mathcal{C}}_X$ which is an extension of $E^1$ by $E_1$ such that $E_2|_D\cong E_{D_2}$. Repeating the above method and applying Proposition~\ref{Hisaturated}, we obtain inductively that there exists $E_i\in\overline{\mathcal{C}}_X$ such that $E_i|_D\cong E_{D_i}$ for any $1\leq i\leq n$. In particular, for any $E_D\in\overline{\mathcal{C}}_D$, there exists $E\in \overline{\mathcal{C}}_X$ such that $E|_D\cong E_D$. Hence by [Theorem~\ref{Thm1}, (2)], the induced homomorphism $\pi(\overline{\mathcal{C}}_D,x)\rightarrow \pi(\overline{\mathcal{C}}_X,x)$ is a closed immersion.

    Consequently, the induced homomorphism $\pi(\overline{\mathcal{C}}_D,x)\rightarrow \pi(\overline{\mathcal{C}}_X,x)$ is an isomorphism.
\end{proof}

\section{The Lefschetz type theorem for fundamental group schemes}

\begin{Definition}
Let $k$ be a field, $X$ a geometrically reduced connected scheme proper over $k$, $x\in X(k)$, $E$ a vector bundle on $X$ of rank $r$. If $k$ is of positive characteristic, then $F_X:X\rightarrow X$ is the absolute Frobenius morphism. Then $E$ is said to be
\begin{itemize}
    \item \textit{numerically flat}, if both $E$ and $E^\vee$ are nef.
    \item \textit{Nori semistable}, if for any smooth proper curve $f:C\rightarrow X$, $f^*E$ is semistable of degree 0 on $C$.
    \item \textit{finite}, if there exist $f(t)\neq g(t)\in\mathbb{N}[t]$ such that $f(E)\cong g(E)$, where
    $$h(E):=\bigoplus_{i=0}^m\bigoplus_{j=1}^{n_i}E^{\otimes i}\text{ for any }h(t)=\sum\limits_{i=0}^mn_it^i\in\mathbb{N}[t].$$
    \item \textit{essentially finite}, if there exists $E_1\hookrightarrow E_2\hookrightarrow F\in\Vect(X)$ such that $E\cong E_2/E_1$, where $E_1,E_2$ are numerically flat and $F$ is finite.
    \item \textit{Frobenius finite}, if there exist $f(t)\neq g(t)\in\mathbb{N}[t]$ such that $\tilde{f}(E)\cong \tilde{g}(E)$, where 
    $$\tilde{h}(E):=\bigoplus\limits_{i=1}^{m}((F_X^i)^*E)^{\oplus n_i} \text{ for any}h(t)=\sum\limits_{i=0}^mn_it^i\in\mathbb{N}[t].$$
    Denote the set of all Frobenius finite vector bundles on $X$ by $FF(X)$.
    \item \textit{essentially Frobenius finite}, if there exists $E_1\hookrightarrow E_2\hookrightarrow F\in\Vect(X)$ such that $E\cong E_2/E_1$, where $E_1,E_2$ are numerically flat and $F\in FF(X)$.
    \item \textit{Frobenius trivial}, if there exists a positive integer $n$ such that $F_X^{n*} E\cong \mathcal{O}_X^{\oplus r}$.
    \item \textit{\'etale trivializable}, if there exists a finite \'etale covering $\phi: P\rightarrow X$ such that $\phi^* E\cong \mathcal{O}_P^{\oplus r}$.
    \item \textit{unipotent}, if there is a filtration
    $0\hookrightarrow E_1\hookrightarrow \cdots\hookrightarrow E_n=E$
    such that $E_{i+1}/E_i\cong\mathcal{O}_X$ for any $i$.
\end{itemize}
We have the following Tannakian categories:
    \begin{itemize}
        \item $\mathcal{C}^{NF}(X)$: objects consist of numerically flat bundles on $X$.
        \item $\mathcal{C}^{N}(X)$: objects consist of essentially finite bundles on $X$.
        \item $\mathcal{C}^{F}(X)$: objects consist of essentially Frobenius finite bundles on $X$.
        \item $\mathcal{C}^{Loc}(X)$: objects consist of Frobenius trivial bundles on $X$.
        \item $\mathcal{C}^{\acute{e}t}(X)$: objects consist of \'etale trivializable bundles on $X$.
        \item $\mathcal{C}^{uni}(X)$: objects consist of unipotent bundles on $X$.
    \end{itemize}
We have the following Tannaka group schemes:
\begin{itemize}
        \item $\pi^{S}(X,x):=\pi(\mathcal{C}^{NF}(X),x)$, called the \textit{S-fundamental group scheme}.
        \item $\pi^{N}(X,x):=\pi(\mathcal{C}^{N}(X),x)$, called the \textit{Nori fundamental group scheme}.
        \item $\pi^{EN}(X,x):=\pi(\overline{\mathcal{C}^{N}(X)},x)$, called the \textit{extended Nori fundamental group scheme}.
        \item $\pi^{F}(X,x):=\pi(\mathcal{C}^{F}(X),x)$, called the \textit{F-fundamental group scheme}.
        \item $\pi^{EF}(X,x):=\pi(\overline{\mathcal{C}^{F}(X)},x)$, called the \textit{extended F-fundamental group scheme}.
        \item $\pi^{Loc}(X,x):=\pi(\mathcal{C}^{Loc}(X),x)$, called the \textit{local fundamental group scheme}.
        \item $\pi^{ELoc}(X,x):=\pi(\overline{\mathcal{C}^{Loc}(X)},x)$, called the \textit{extended local fundamental group scheme}.
        \item $\pi^{\acute{e}t}(X,x):=\pi(\mathcal{C}^{\acute{e}t}(X),x)$, called the \textit{\'etale fundamental group scheme}.
        \item $\pi^{E\acute{e}t}(X,x):=\pi(\overline{\mathcal{C}^{\acute{e}t}(X)},x)$, called the \textit{extended \'etale fundamental group scheme}.
        \item $\pi^{uni}(X,x):=\pi(\mathcal{C}^{uni}(X),x)$, called the \textit{unipotent fundamental group scheme}.
    \end{itemize}
\end{Definition}

\begin{Remark}\label{tannakianrelation}
    Let $k$ be a field, $X$ a geometrically reduced connected scheme proper over $k$. Then 
    \begin{enumerate}
        \item $\mathcal{C}^{NF}(X)$ and $\mathcal{C}^{uni}(X)$ are saturated categories, i.e. $\overline{\mathcal{C}^{NF}(X)}=\mathcal{C}^{NF}(X)$ and $\overline{\mathcal{C}^{uni}(X)}=\mathcal{C}^{uni}(X)$.
        \item $\mathcal{C}^{*}(X)$ and $\overline{\mathcal{C}^{*}(X)}$ are Tannakian subcategories of $\mathcal{C}^{NF}(X)$ for $*\in\{N,F,Loc,\acute{e}t,uni\}$.
    \end{enumerate}
\end{Remark}

\begin{Lemma}[{\cite[Proposition~5.7]{LiTiBa26}}]\label{iffnumericallyflat}
    Let $k$ be a field, $K/k$ a field extension, $X$ a geometrically reduced connected scheme proper over $k$, $x_K\in X_K(K)$ lying over $x\in X(k)$, $E\in\Vect(X)$. Then $E\in\mathcal{C}^{NF}(X)$ iff $E\otimes_k K\in\mathcal{C}^{NF}(X_K)$.
\end{Lemma}

\begin{Lemma}[{\cite[Proposition~5.8]{LiTiBa26}}]\label{LiTianbasechange}
    Let $k$ be a field, $K/k$ a separable extension, $X$ a geometrically reduced connected scheme proper over $k$, $x_K\in X_K(K)$ lying over $x\in X(k)$. Then the natural homomorphism $$\pi^S(X_K,x_K)\rightarrow \pi^S(X,x)_K$$
    is an isomorphism.
\end{Lemma}

\begin{Lemma}[{\cite[Lemma~4.5 \& Proposition~5.4]{LiTiKu26}}]\label{pullbackall}
    Let $k$ be a field, $f:X\rightarrow S$ a morphism of geometrically reduced connected schemes proper over $k$. Then for any $E\in\mathcal{C}^{*}(S)$, we have $f^*E\in\mathcal{C}^{*}(X)$ where $*\in\{S,N,EN,F, EF,Loc,ELoc,\acute{e}t,E\acute{e}t,uni\}$.
\end{Lemma}

\begin{Proposition}\label{isodescend}
    Let $k$ be a field, $f:X\rightarrow S$ a morphism of geometrically reduced connected schemes proper over $k$, $x\in X(k)$ lying over $s\in S(k)$. If the induced homomorphism $\pi^S(X,x)\rightarrow \pi^S(S,s)$ is an isomorphism, then
    \begin{enumerate}
        \item The induced homomorphism $\pi^N(X,x)\rightarrow \pi^N(S,s)$ is an isomorphism.
        \item The induced homomorphism $\pi^{EN}(X,x)\rightarrow \pi^{EN}(S,s)$ is an isomorphism.
        \item The induced homomorphism $\pi^{\acute{e}t}(X,x)\rightarrow \pi^{\acute{e}t}(S,s)$ is an isomorphism.
        \item The induced homomorphism $\pi^{E\acute{e}t}(X,x)\rightarrow \pi^{E\acute{e}t}(S,s)$ is an isomorphism.
        \item The induced homomorphism $\pi^{uni}(X,x)\rightarrow \pi^{uni}(S,s)$ is an isomorphism.
        \item If $\chara k>0$, then the induced homomorphism $\pi^{F}(X,x)\rightarrow \pi^{F}(S,s)$ is an isomorphism.
        \item If $\chara k>0$, then the induced homomorphism $\pi^{EF}(X,x)\rightarrow \pi^{EF}(S,s)$ is an isomorphism.
        \item If $\chara k>0$, then the induced homomorphism $\pi^{Loc}(X,x)\rightarrow \pi^{Loc}(S,s)$ is an isomorphism.
        \item If $\chara k>0$, then the induced homomorphism $\pi^{ELoc}(X,x)\rightarrow \pi^{ELoc}(S,s)$ is an isomorphism.
    \end{enumerate}
\end{Proposition}

\begin{proof}
    (1) Since the induced homomorphism $\pi^S(X,x)\rightarrow \pi^S(S,s)$ is an isomorphism, for any $E\in\mathcal{C}^{N}(X)$, there exists $F\in\mathcal{C}^{NF}(S)$ such that $E\cong f^*F$. 
    
    Suppose $E$ is a finite bundle on $X$, then there exist distinct $g(t),h(t)\in\mathbb{N}[t]$ such that $f^*g(F)\cong g(E)\cong h(E)\cong f^*h(F)$. Since the functor $f^*:\mathcal{C}^{NF}(S)\rightarrow \mathcal{C}^{NF}(X)$ is fully faithful, there exists an isomorphism $g(F)\cong h(F)$. So $F$ is a finite bundle on $S$.

    Suppose $E\in \mathcal{C}^{N}(X)$, then $E\cong E_1/E_2$, where $E_2\hookrightarrow E_1\hookrightarrow E'$, $E_1,E_2\in\mathcal{C}^{NF}(X)$ and $E'$ is a finite bundle on $X$. Since the functor $f^*:\mathcal{C}^{NF}(S)\rightarrow \mathcal{C}^{NF}(X)$ is fully faithful, there exists a filtration $F_{2}\hookrightarrow F_{1}\hookrightarrow F'$ in $\mathcal{C}^{NF}(S)$ with $F'$ being a finite bundle on $S$ such that $E'\cong f^*F'$ and $E_i\cong f^*F_i$, $i=1,2$. It follows that $F\cong F_1/F_2\in\mathcal{C}^N(S)$. 
    
    Hence the induced homomorphism $\pi^{N}(X,x)\rightarrow \pi^{N}(S,s)$ is an isomorphism by Proposition~\ref{isosubcate}.

    (2) It follows by (1) and Proposition~\ref{isoSerre}.
    
    (3) Since the induced homomorphism $\pi^S(X,x)\rightarrow \pi^S(S,s)$ is an isomorphism, for any $E\in\mathcal{C}^{\acute{e}t}(X)$, there exists $F\in\mathcal{C}^{NF}(S)$ such that $E\cong f^*F$. Consider the following commutative diagram
    \[\begin{tikzcd}
        \mathcal{C}^{NF}(S)\arrow[r,"f^*"]&\mathcal{C}^{NF}(X)\\
        \langle F \rangle_{\mathcal{C}^{NF}(S)}\arrow[r,"f^*"]\arrow[u,hook]&\langle f^*F \rangle_{\mathcal{C}^{NF}(X)}\arrow[u,hook]
    \end{tikzcd}\quad\quad\begin{tikzcd}
        \pi^{S}(X,x)\arrow[r,"\cong"]\arrow[d,two heads]&\pi^{S}(S,s)\arrow[d,two heads]\\
        \pi(\langle f^*F \rangle_{\mathcal{C}^{NF}(X)},x)\arrow[r]&\pi(\langle F \rangle_{\mathcal{C}^{NF}(S)},s)
    \end{tikzcd}\]
    It follows that the homomorphism $\pi(\langle f^*F \rangle_{\mathcal{C}^{NF}(X)},x)\twoheadrightarrow \pi(\langle F \rangle_{\mathcal{C}^{NF}(S)},s)$ induced by $f^*$ is faithfully flat. Since $f^*F\cong E\in\mathcal{C}^{\acute{e}t}(X)$, $\pi(\langle F \rangle_{\mathcal{C}^{\acute{e}t}(S)},s)$ is an \'etale finite group scheme. Then we obtain $F \in\mathcal{C}^{\acute{e}t}(S)$. Hence the induced homomorphism $\pi^{\acute{e}t}(X,x)\rightarrow \pi^{\acute{e}t}(S,s)$ is an isomorphism by Proposition~\ref{isosubcate}.

    (4) It follows by (3) and Proposition~\ref{isoSerre}.

    (5) For any $E\in\mathcal{C}^{uni}(X)$, there exists a filtration on $X$:
    $0=E_{0}\hookrightarrow E_{1}\hookrightarrow \cdots\hookrightarrow E_{n}=E$, 
    where $E_{{i+1}}/E_{{i}}\cong \mathcal{O}_X$ for any $0\leq i\leq n-1$. Note that $E_{i}\in\mathcal{C}^{NF}(X)$ for any $0\leq i\leq n$. Since the induced homomorphism $\pi^S(X,x)\rightarrow \pi^S(S,s)$ is an isomorphism, there exists a filtration in $\mathcal{C}^{NF}(S)$:
    $$0=F_{0}\hookrightarrow F_{1}\hookrightarrow \cdots\hookrightarrow F_{n}=F$$
    such that $E_{{i}}\cong f^*F_i$ and $E_{i+1}/E_i\cong f^*(F_{i+1}/F_i)\cong \mathcal{O}_X$ for any $i$. Since the functor $f^*:\mathcal{C}^{NF}(S)\rightarrow \mathcal{C}^{NF}(X)$ is fully faithful, we have $F_{i+1}/F_i\cong \mathcal{O}_S$ for any $0\leq i\leq n-1$. It follows that $F\in\mathcal{C}^{uni}(S)$. Hence the induced homomorphism $\pi^{uni}(X,x)\rightarrow \pi^{uni}(S,s)$ is an isomorphism by Proposition~\ref{isosubcate}.
    
    (6) Since the induced homomorphism $\pi^S(X,x)\rightarrow \pi^S(S,s)$ is an isomorphism, for any $E\in\mathcal{C}^{F}(X)$, there exists $F\in\mathcal{C}^{NF}(S)$ such that $E\cong f^*F$. 
    
    Suppose $E\in FF(X)$, then there exist integers $m>n>0$ such that $F_X^{m*}E\cong F_X^{n*}E$. Since the functor $f^*:\mathcal{C}^{NF}(S)\rightarrow \mathcal{C}^{NF}(X)$ is fully faithful, there exists an isomorphism $F_S^{m*}F\cong F_S^{n*}F$. So $F\in FF(S)$.

    Suppose $E\in \mathcal{C}^{F}(X)$, then $E\cong E_1/E_2$, where $E_2\hookrightarrow E_1\hookrightarrow E'$, $E_1,E_2\in\mathcal{C}^{NF}(X)$ and $E'\in FF(X)$. Since the functor $f^*:\mathcal{C}^{NF}(S)\rightarrow \mathcal{C}^{NF}(X)$ is fully faithful, there exists a filtration $F_{2}\hookrightarrow F_{1}\hookrightarrow F'$ in $\mathcal{C}^{NF}(S)$ with $F'\in FF(S)$ such that $E'\cong f^*F'$ and $E_i\cong f^*F_i$, $i=1,2$. Then $F\cong F_1/F_2\in\mathcal{C}^F(S)$. 
    
    Hence the induced homomorphism $\pi^{F}(X,x)\rightarrow \pi^{F}(S,s)$ is an isomorphism by Proposition~\ref{isosubcate}.

    (7) It follows by (6) and Proposition~\ref{isoSerre}.
    
    (8) For any $E\in\mathcal{C}^{Loc}(X)$ of rank $r$, there exists an integer $n>0$ such that $F_X^{n*}E\cong \mathcal{O}_X^{\oplus r}$. Since the induced homomorphism $\pi^S(X,x)\rightarrow \pi^S(S,s)$ is an isomorphism, there exists $F\in\mathcal{C}^{NF}(S)$ such that $E\cong f^*F$ and $F_X^{n*} E\cong f^*(F_S^{n*}F)\cong \mathcal{O}_X^{\oplus r}$. Since the functor $f^*:\mathcal{C}^{NF}(S)\rightarrow \mathcal{C}^{NF}(X)$ is fully faithful, then $F_S^{n*}F\cong \mathcal{O}_S^{\oplus r}$, i.e. $F\in \mathcal{C}^{Loc}(S)$. Hence the induced homomorphism $\pi^{Loc}(X,x)\rightarrow \pi^{Loc}(S,s)$ is an isomorphism by Proposition~\ref{isosubcate}.

    (9) It follows by (8) and Proposition~\ref{isoSerre}.
\end{proof}

\subsection{The case of characteristic 0}

\begin{Lemma}[{\cite[Theorem~11.4]{Lan11}}]\label{Lancomplex}
    Let $X$ be a complex projective manifold of dimension $d$, $D\subsetneq X$ a smooth ample effective divisor, $x\in D(k)$. Then
    \begin{enumerate}
        \item If $d\geq 2$, then the induced homomorphism $\pi^S(D,x)\rightarrow \pi^S(X,x)$ is faithfully flat.
        \item If $d\geq 3$, then the induced homomorphism $\pi^S(D,x)\rightarrow \pi^S(X,x)$ is an isomorphism.
    \end{enumerate}
\end{Lemma}

\begin{Proposition}\label{App0}
    Let $k$ be a field of characteristic $0$, $X$ a smooth projective variety over $k$ of dimension $d$, $D\subsetneq X$ a smooth ample effective divisor, $x\in D(k)$, $*\in\{S,\acute{e}t,E\acute{e}t,uni\}$. Then
    \begin{enumerate}
        \item If $d\geq 2$, then the induced homomorphism $\pi^*(D,x)\rightarrow \pi^{*}(X,x)$ is faithfully flat.
        \item If $d\geq 3$, then the induced homomorphism $\pi^*(D,x)\rightarrow \pi^*(X,x)$ is an isomorphism.
    \end{enumerate}
\end{Proposition}

\begin{proof}
    (1) By Lemma~\ref{LiTianbasechange}, Lemma~\ref{Lancomplex} and Lefschetz principle, we have the following commutative diagram
    \[
        \begin{tikzcd}
        \pi^S(D_{\bar{k}},\bar{x})\arrow[r,two heads]\arrow[d,"\cong"]&\pi^S(X_{\bar{k}},\bar{x})\arrow[d,"\cong"]\\
        \pi^S(D,x)_{\bar{k}}\arrow[r,two heads]\arrow[d,two heads]&\pi^S(X,x)_{\bar{k}}\arrow[d,two heads]\\
            \pi^*(D,x)_{\bar{k}}\arrow[r]&\pi^*(X,x)_{\bar{k}}
        \end{tikzcd}
    \]
    It follows that $\pi^*(D,x)_{\bar{k}}\rightarrow\pi^*(X,x)_{\bar{k}}$ is faithfully flat. Hence the natural homomorphism $\pi^*(D,x)\rightarrow \pi^{*}(X,x)$ is faithfully flat by Lemma~\ref{isofield}.

    (2) By Lemma~\ref{LiTianbasechange}, Lemma~\ref{Lancomplex} and Lefschetz principle, we have the following commutative diagram
    \[
        \begin{tikzcd}
        \pi^S(D_{\bar{k}},\bar{x})\arrow[r,"\cong"]\arrow[d,"\cong"]&\pi^S(X_{\bar{k}},\bar{x})\arrow[d,"\cong"]\\
            \pi^S(D,x)_{\bar{k}}\arrow[r]&\pi^S(X,x)_{\bar{k}}
        \end{tikzcd}
    \]
    It follows that $\pi^S(D,x)_{\bar{k}}\rightarrow\pi^S(X,x)_{\bar{k}}$ is an isomorphism. Hence the natural homomorphism $\pi^S(D,x)\rightarrow \pi^{S}(X,x)$ is an isomorphism by Lemma~\ref{isofield}. Then by Proposition~\ref{isodescend}, it follows immediately.
\end{proof}

\subsection{The case of characteristic $p>0$}

\begin{Lemma}[{\cite[Theorem~10.4 \& Corollary~10.7]{Lan11}}]\label{Lan11}
    Let $k$ be an algebraically closed field of characteristic $p>0$, $X$ a smooth projective variety over $k$ of dimension $d\geq 3$, $H$ an ample divisor on $X$, $\alpha$ a nonnegative integer such that $T_X(\alpha H)$ is globally generated. Let $D\subsetneq X$ be a smooth ample effective divisor such that $D-\alpha H$ is ample and $DH^{d-1}> \max(p\alpha H^d,(d+1)\alpha H^d-K_X H^{d-1})$, $x\in D(k)$, $*\in\{S,N,\acute{e}t\}$. Then the induced homomorphism $\pi^*(D,x)\rightarrow \pi^*(X,x)$ is an isomorphism.
\end{Lemma}

\begin{Proposition}\label{app+}
    Let $k$ be a perfect field of characteristic $p>0$, $X$ a smooth projective variety over $k$ of dimension $d\geq 3$, $H$ an ample divisor on $X$, $\alpha$ a nonnegative integer such that $T_X(\alpha H)$ is globally generated. Let $D\subsetneq X$ be a smooth ample effective divisor such that $D-\alpha H$ is ample and $DH^{d-1}> \max(p\alpha H^d,(d+1)\alpha H^d-K_X H^{d-1})$, $x\in D(k)$, $*\in\{S,N,EN,F, EF,Loc,ELoc,\acute{e}t,E\acute{e}t,uni\}$. Then the induced homomorphism $\pi^*(D,x)\rightarrow \pi^*(X,x)$ is an isomorphism.
\end{Proposition}

\begin{proof}
    By Lemma~\ref{LiTianbasechange} and Lemma~\ref{Lan11}, we have the following commutative diagram
    \[
        \begin{tikzcd}
        \pi^S(D_{\bar{k}},\bar{x})\arrow[r,"\cong"]\arrow[d,"\cong"]&\pi^S(X_{\bar{k}},\bar{x})\arrow[d,"\cong"]\\
            \pi^S(D,x)_{\bar{k}}\arrow[r]&\pi^S(X,x)_{\bar{k}}
        \end{tikzcd}
    \]
    It follows that $\pi^S(D,x)_{\bar{k}}\rightarrow\pi^S(X,x)_{\bar{k}}$ is an isomorphism. Hence the induced homomorphism $\pi^S(D,x)\rightarrow \pi^{S}(X,x)$ is an isomorphism by Lemma~\ref{isofield}. Then by Proposition~\ref{isodescend}, it follows immediately.
\end{proof}

\begin{Remark}
    Let $k$ be an algebraically closed field of characteristic $p>0$, $X$ a normal connected projective variety over $k$ of dimension $d\geq 3$, $F_X:X\rightarrow X$ the absolute Frobenius morphism. By \cite[Theorem~2.1]{BhJo14}, if $E\in\Vect(X)$ is trivial over an ample divisor, then $E$ is trivialized by a torsor for a finite connected $k$-group scheme. In particular, $F_X^{n}E\cong \mathcal{O}_X^{\oplus r}$ for $n\gg0$, which implies $E\in\mathcal{C}^{Loc}(X)$ and maybe $E\not\cong \mathcal{O}_X^{\oplus r}$. So the positivity of divisor in Lemma~\ref{Lan11} and Proposition~\ref{app+} is of great importance.
\end{Remark}

\subsection{The case of varieties liftable to $W_2(k)$}

\begin{Lemma}[{\cite[Corollary~11.3]{Lan11}}]\label{Lanlifting}
    Let $k$ be a perfect field of characteristic $p>0$, $X$ a smooth variety proper over $k$ of dimension $d$ liftable to $W_2(k)$, $D\subsetneq X$ a smooth ample effective divisor, $x\in D(k)$. Then
    \begin{enumerate}
        \item If $d\geq 2$, then the induced homomorphism $\pi^S(D,x)\rightarrow \pi^S(X,x)$ is faithfully flat.
        \item If $d\geq 3$ and $p\geq 3$, then the induced homomorphism $\pi^S(D,x)\rightarrow \pi^S(X,x)$ is an isomorphism.
    \end{enumerate}
\end{Lemma}

\begin{Proposition}\label{applift}
    Let $k$ be a perfect field of characteristic $p>0$, $X$ a smooth variety proper over $k$ of dimension $d$ liftable to $W_2(k)$, $*\in\{N,EN,F, EF,Loc,ELoc,\acute{e}t,E\acute{e}t,uni\}$, $D\subsetneq X$ a smooth ample effective divisor, $x\in D(k)$. Then
    \begin{enumerate}
        \item If $d\geq 2$, then the induced homomorphism $\pi^*(D,x)\rightarrow \pi^*(X,x)$ is faithfully flat.
        \item If $d\geq 3$ and $p\geq 3$, then the induced homomorphism $\pi^*(D,x)\rightarrow \pi^*(X,x)$ is an isomorphism.
    \end{enumerate}
\end{Proposition}

\begin{proof}
    (1) By Lemma~\ref{Lanlifting}, we have the following commutative diagram
    \[
        \begin{tikzcd}
        \pi^S(D,x)\arrow[r,two heads]\arrow[d,two heads]&\pi^S(X,x)\arrow[d,two heads]\\
            \pi^*(D,x)\arrow[r]&\pi^*(X,x)
        \end{tikzcd}
    \]
    It follows that the natural homomorphism $\pi^*(D,x)\rightarrow \pi^*(X,x)$ is faithfully flat.

    (2) By Lemma~\ref{LiTianbasechange} and Lemma~\ref{Lanlifting}, we have the following commutative diagram
    \[
        \begin{tikzcd}
        \pi^S(D_{\bar{k}},\bar{x})\arrow[r,"\cong"]\arrow[d,"\cong"]&\pi^S(X_{\bar{k}},\bar{x})\arrow[d,"\cong"]\\
            \pi^S(D,x)_{\bar{k}}\arrow[r]&\pi^S(X,x)_{\bar{k}}
        \end{tikzcd}
    \]
    It follows that $\pi^S(D,x)_{\bar{k}}\rightarrow\pi^S(X,x)_{\bar{k}}$ is an isomorphism. Hence the natural homomorphism $\pi^S(D,x)\rightarrow \pi^{S}(X,x)$ is an isomorphism by Lemma~\ref{isofield}. Then by Proposition~\ref{isodescend}, it follows immediately.
\end{proof}



\begin{thebibliography}{99}

\bibitem{AdAm25} P. Adroja and S. Amrutiya, \emph{On an extension of Nori and local fundamental group schemes}, Comm. Algebra {\bf 53} (2025), no.~10, 4241--4255.

\bibitem{AmBi10} S. Amrutiya, I. Biswas, \emph{On the F-fundamental group scheme}, Bull. Sci. Math. {\bf 134} (2010), no. 5, 461--474.

\bibitem{BiHo07} I. Biswas, Y. I. Holla, \emph{Comparison of fundamental group schemes of a projective variety and an ample hypersurface}, J. Algebraic Geom. 16 (2007), no. 3, 547–597.

\bibitem{BhJo14} B. Bhatt and A.~J. de~Jong, \emph{Lefschetz for local Picard groups}, Ann. Sci. \'Ec. Norm. Sup\'er. (4) {\bf 47} (2014), no.~4, 833--849.

\bibitem{EHS07} H. Esnault, P. H. Hai, X. Sun, \emph{On Nori's fundamental group scheme}, Geometry and dynamics of groups and spaces, 377--398, Progr. Math., 265, Birkhäuser, Basel, 2007.

\bibitem{FuLa22} M. Fulger, A. Langer, \emph{Positivity vs. slope semistability for bundles with vanishing discriminant}, J. Algebra. {\bf 609} (2022), 657--687.

\bibitem{Gro60} A. Grothendieck, \emph{Revêtements étales et groupe fondamental}, Séminaire de Géométrie Algébrique du Bois-Marie, (SGA 1), 1960/61, Lectures Notes in Mathematics, Vol. 224, Springer-Verlag, Berlin, (1971).

\bibitem{Har70} R. Hartshorne, {\it Ample subvarieties of algebraic varieties}, Lecture Notes in Mathematics, Vol. 156, Springer, Berlin-New York, 1970.

\bibitem{Har94} R. Hartshorne, \emph{Generalized divisors on Gorenstein schemes}, $K$-Theory {\bf 8} (1994), no.~3, 287--339.

\bibitem{HuLe10} D. Huybrechts and M. Lehn, \emph{The geometry of moduli spaces of sheaves}, second edition, Cambridge Mathematical Library, Cambridge Univ. Press, Cambridge, 2010.

\bibitem{Lan11} A. Langer, \emph{On the S-fundamental group scheme}, Ann. Inst. Fourier (Grenoble) {\bf 61} (2011), no.~5, 2077--2119 (2012).

\bibitem{LiTiBa26} L. Li, N. Tian, \emph{Base change of fundamental group schemes}, https://doi.org/10.48550/arXiv.2602.11110.

\bibitem{LiTiKu26} L. Li, N. Tian, \emph{The K\"unneth formula of fundamental group schemes}, https://doi.org/10.48550/arXiv.2602.14207.

\bibitem{MeSu08} V.~B. Mehta and S. Subramanian, \emph{Some remarks on the local fundamental group scheme}, Proc. Indian Acad. Sci. Math. Sci. {\bf 118} (2008), no.~2, 207--211.

\bibitem{Mil12} J. S. Milne, \emph{Basic theory of affine group schemes}, 2012, Available at www.jmilne.org/math/.

\bibitem{Nor76} M. V. Nori, \emph{On the representations of the fundamental group}, Compositio Math. {\bf 33} (1976), no.~1, 29--41.

\bibitem{Nor82} M. V. Nori, \emph{The fundamental group-scheme}, Proc. Indian Acad. Sci. Math. Sci. {\bf 91} (1982), no.~2, 73--122.

\bibitem{Ota17} S. Otabe, \emph{An extension of Nori fundamental group}, Comm. Algebra {\bf 45} (2017), no.~8, 3422--3448.

\end{thebibliography}
\end{document}